\documentclass[oneside,12pt]{amsart}
\usepackage{eurosym}
\usepackage{comment}
\usepackage{amsfonts}
\usepackage{amsmath,amssymb,amsthm,color}
\usepackage{amsopn}
\usepackage{latexsym}
\usepackage{float}
\usepackage{bm}
\usepackage{xcolor}
\usepackage[utf8]{inputenc}
\usepackage{hhline}

\renewcommand{\d}{{\rm d}}
\newcommand{\mmi}{\mathfrak i }
\newcommand{\mr}{\mathfrak r }

\newcommand{\mg}{\mathfrak g }

\newcommand{\ms}{\mathfrak s }

\newcommand{\mmu}{\mathfrak u }

\newcommand{\mk}{\mathfrak k }

\newcommand{\mh}{\mathfrak h }

\newcommand{\mmp}{\mathfrak p }

\newcommand{\so}{\mathfrak{so} }

\newcommand{\gl}{\mathfrak{gl} }

\newcommand{\lela}{ g(}
\newcommand{\rira}{)}

\newcommand{\lra}{\longrightarrow}

\newcommand{\bs}{\backslash}

\newcommand{\GL}{\operatorname{GL}}

\renewcommand{\Im}{\rm Im\,}
\newcommand{\R}{\mathbb R}

\newcommand{\N}{\mathbb N}
\newcommand{\Z}{\mathbb Z}

\newcommand{\ts}{\theta^\sharp}
\newcommand{\rmP}{{\rm P}}

\DeclareMathOperator{\Sim}{Sim}
\DeclareMathOperator{\End}{End}

\DeclareMathOperator{\Aut}{Aut}

\DeclareMathOperator{\ad}{ad}

\DeclareMathOperator{\tr}{tr}

\numberwithin{equation}{section}

 \newtheorem{teo}{Theorem}[section]
 \newtheorem{pro}[teo]{Proposition}
  
 \newtheorem{cor}[teo]{Corollary}
 \newtheorem{lm}[teo]{Lemma}
 \newtheorem{defi}[teo]{Definition}

 \theoremstyle{definition}
 
 \newtheorem{remark}[teo]{Remark}  
  \newtheorem{question}[teo]{Question}

\newcommand{\nc}{\newcommand}
\nc{\Iso}{\operatorname{Iso}}
\nc{\Id}{\operatorname{Id}}
 \nc{\iso}{\mathfrak{iso}}
 \nc{\sso}{\mathfrak{so}}
\nc{\Ad}{\operatorname{Ad}} 
\nc{\Sym}{\mathrm{Sym}}
 \nc{\pr}{\operatorname{pr}} 
 \nc{\Dera}{\operatorname{Dera}} 
 \nc{\Auto}{\operatorname{Auto}}
 \nc{\noi}{\noindent}
 \nc{\SO}{\operatorname{\rm SO}}

\title{The characteristic group of Lie LCP manifolds}

\author{Viviana del Barco}
\address{V.~del Barco: Instituto de Matemática, Estatística e Computação Científica, Universidade Estadual de Campinas,  Rua Sergio Buarque de Holanda, 651, Cidade Universitaria Zeferino Vaz, 13083-859, Campinas, São Paulo, Brazil.}
\email{delbarc@ime.unicamp.br}

\author{Andrei Moroianu}
\address{A.~Moroianu: Université Paris-Saclay, CNRS,  Laboratoire de mathématiques d'Orsay, 91405, Orsay, France, and Institute of Mathematics “Simion Stoilow” of the Romanian Academy, 21 Calea Grivitei, 010702 Bucharest, Romania}
\email{andrei.moroianu@math.cnrs.fr}


\begin{document}

\begin{abstract}
   The (reduced) characteristic group of a locally conformally product manifold is obtained by restricting the action of its fundamental group to the non-flat factor of the universal cover, and taking the connected component of the identity in the closure of this restriction. It was shown by Kourganoff that this group is abelian, but it is currently unknown whether it is simply connected, or might have compact (toric) factors. This question is crucial for a better understanding of LCP structure, as shown recently by B. Flamencourt. In this paper we consider Lie LCP structures (which are defined on quotients of simply connected Lie groups by lattices) and show that the reduced characteristic group of any Lie LCP manifold $\Gamma\bs G$ is contained in the radical of $G$, so in particular is simply connected.
\end{abstract}

\maketitle

\section{Introduction}

Locally conformally product (LCP) structures have been introduced about a decade ago as a result of a conjecture of F.~Belgun and the second author stating that a simply connected Riemannian manifold $(\tilde M,h)$ having a free, discrete co-compact group $\Gamma$ of similarities of $h$, not all being isometries, is either irreducible or flat \cite{BM2016}. A counterexample was found by Matveev-Nicolayevsky \cite{MN2015}, who also proved a couple of years later \cite{MN2017} that in the analytic category, all counterexamples $(\tilde M,h)$ can be obtained as Riemannian products of some flat space $\R^q$ with an irreducible Riemannian manifold $(N,g_N)$. Shortly after, M.~Kourganoff proved that the same is true in the smooth category \cite[Theorem 1.5]{Kourganoff}. The corresponding structures on the quotient manifolds $\Gamma\bs \tilde M$ are called LCP structures.

By definition, an LCP manifold is thus a compact manifold $M$ whose universal cover $\tilde M$ has a non-flat, reducible Riemannian metric together with a free, discrete co-compact group $\Gamma$ of similarities of $h$, not all being isometries. The metric $h$ induces a conformal structure $c$ on $M$, and since similarities are affine transformations, the Levi-Civita connection of $h$ projects to a torsion-free connection $D$ on $M$ preserving the conformal class $c$. Such connections are called Weyl connections \cite{Gaud94,We23}. A Weyl connection on $(M,c)$ is called closed (resp.~exact) if it is locally (resp.~globally) the Levi-Civita connection of a metric in $c$. One can thus equivalently define LCP manifolds as compact conformal manifolds carrying a closed, non-exact, non-flat Weyl structure. 

By the above mentioned result of Kourganoff, the universal cover $\tilde M$ of an LCP manifold $(M,c,D)$, endowed with the Riemannian metric $h$ in the pull-back of the conformal class $c$ whose Levi-Civita connection is the pull-back of $D$, has a de Rham decomposition with exactly two factors: a flat space $\R^q$ with $q\ge 1$, and an irreducible incomplete Riemannian manifold $(N,g_N)$. The fundamental group $\Gamma$ of $M$  acts by similarities of $h$ so preserves the de Rham decomposition. Thus each element $\gamma\in\Gamma$ can be written $\gamma=(\gamma_1,\gamma_2)$ with $\gamma_1\in\Sim(\R^q)$ and $\gamma_2\in\Sim(N,g_N)$, where $\Sim$ denotes the group of similarities. 

A crucial object for the study of LCP manifolds is defined as follows: one denotes by $\rmP\subset \Sim(N,g_N)$ the projection of $\Gamma$, and by $\bar\rmP_0$ the identity component of its closure in $\Sim(N,g_N)$. The group $\R^q\times \bar\rmP_0\subset \Sim(\R^q)\times \Sim(N,g_N)$ is called the {\em characteristic group} of the LCP structure \cite{Fl25}, and $\bar\rmP_0$ is called the {\em reduced characteristic group}. A highly non-trivial result of Kourganoff \cite{Kourganoff} is that these groups are abelian. As a consequence, B.~Flamencourt has shown  that for every LCP manifold, the similarity factors of the elements of $\Gamma$ are algebraic numbers (see \cite{Fl24}).
A central question raised by Flamencourt \cite{Fl25} is whether the characteristic group is always simply connected. Indeed, if this holds, then a reduction procedure allows one a possible classification attempt for LCP structures, or at least a better understanding of them (see \cite[Theorems 4.4 and 4.10] {Fl25}).

It is this question that we will study in this paper, in the particular framework of Lie LCP manifolds, which are LCP manifolds whose universal cover is a Lie group $G$ and such that the fundamental group acts on $G$ by left translations (see Definition \ref{lielcp}). In this case the LCP structure can be described in algebraic terms on the Lie algebra $\mg$ of $G$, which turns out to have a semi-direct product structure $\mg=\mmu\rtimes_\alpha\mh$ where $\mmu$ is an abelian ideal of $\mg$, and $\mh$ is a non-unimodular Lie algebra acting on $\mmu$ via a conformal representation $\alpha$ (see \S\ref{sec:LCPLieAlg} below or \cite[Corollary 4.10]{dBM24}).

Correspondingly, the Lie group $G$ can be written as a semi-direct product $\R^q\rtimes H$, and the reduced characteristic group of $M=\Gamma\bs G$ is obtained by projecting $\Gamma$ on $H$ and taking the identity component of the identity in the closure of this projection. Our main result (Theorem \ref{main}) states that this group is contained in the radical of $G$. The strategy of the proof is the following. Using a Levi decomposition of $H$ as $H=R_0\rtimes S$, where $R_0$ is the radical and $S$ the semisimple part, one obtains a Levi decomposition of $G=\R^q\rtimes H$ by writing 
$$G=\R^q\rtimes(R_0\rtimes S)=
(\R^q\rtimes R_0)\rtimes S=:R\rtimes S.$$
We then further decompose $S=S_1\times K$, where $K$ is the maximal compact factor of $S$ acting trivially on $R$ and $S_1$ is the product of the remaining simple factors of $S$. This gives the following decomposition of $G$:
$$G=(\R^q\rtimes R_0)\rtimes S_1\times K.$$
In Proposition \ref{p=q} we show, using a result of Ragunathan  on intersection of lattices with the radical, that the reduced characteristic group of $\Gamma\bs G$ projects trivially on $S_1$, and is isomorphic to the reduced characteristic group of $\Sigma\bs (R\times K)$, where $\Sigma:=\Gamma\cap(R\times K)$. This allows one to reduce the study to the case where $S_1$ is trivial, i.e. $G=R\times K$ is a direct product of a solvable Lie group and a compact Lie group.

Using the Tits alternative, we then show in Lemma \ref{lm:g2va} that the projection of $\Gamma$ on the compact factor $K$ is virtually abelian. On the other hand, some preliminary algebraic results on lattices on direct sums of vector spaces (Lemma \ref{5.1} and Proposition \ref{pro:findsol}), show that the intersection of $\Gamma$ with the characteristic group of $(\Gamma\cap R)\bs R$ has a subgroup of finite index contained in its derived group. Using this, we realise the projection of $\Gamma$ onto $R_0\times K$ as the graph of a morphism from $R_0$ to $K$ with finite image, therefore allowing us to conclude that the characteristic group of the Lie LCP structure is contained in the radical $R$ of $G$.  Consequently, this group has to be simply connected, as any Lie subgroup of a solvable simply connected Lie group, thus answering Flamencourt's question in the Lie LCP setting. 

A brief account of the contents of the paper is as follows: Section 2 introduces notations and summarizes previous results on (Lie) LCP structures. Section 3 makes an insight on the geometry of semi-direct products of Riemannian Lie groups, when the action is given by conformal maps. In particular, we explicit the de Rham decomposition in Kourganoff's result of Lie LCP manifolds. Section 4 revisits the definition of the characteristic group given in Section 2, now in the case of Lie LCP manifolds, and introduces technical results that will be used in Section 5. It also includes a simple proof that the characteristic group of a Lie LCP {\em solvmanifold} is simply connected. Section 5 addresses the general case; the whole section is devoted to the proof of Theorem \ref{main} using the aforementioned techniques. 

As a consequence,  we obtain two important results on the characteristic group of any Lie LCP manifold: it is simply connected (Corollary \ref{cor:sc}) and is contained in the commutator of the universal cover (Corollary \ref{cor:comm}). The latter was a phenomena observed in all examples in \cite{AdBM, dBM24}, but no proof was given up to now. In those examples it also holds that the characteristic group of a Lie LCP manifold is contained in the nilradical of the universal cover and also, it does not depend on the lattice $\Gamma\subset G$, once the flat factor is fixed. It is still an open question whether these facts hold in general.
\medskip

\noindent {\sc Acknowledgements:} VdB is partially supported by FAPESP grants 2023/15089-9 and 2024/19272-5.  AM is partially supported by the PNRR-III-C9-2023-I8 grant CF 149/31.07.2023 {\em Conformal Aspects of Geometry and Dynamics}. Both authors are partially supported by MATHAMSUD Regional Program 24-MATH-12.

\section{Preliminaries}

\subsection{LCP Lie algebras}\label{sec:LCPLieAlg}

Let $\mg$ denote a finite dimensional real Lie algebra and let $g$ be an inner product on $\mg$. We denote by $G$ the connected and simply connected Lie group with Lie algebra $\mg$, and by the same letter $g$ the left-invariant Riemannian metric on $G$ induced by the inner product $g$.

The Levi-Civita connection $\nabla^g$ of the Riemannian manifold $(G,g)$ preserves left-invariant vector fields: if $X,Y$ are left-invariant vector fields on $G$ then so is $\nabla^g_XY$. Hence, $\nabla^g$ can be identified with the linear map $\nabla^g:\mg\to\so(\mg)$ which, by Koszul's formula, satisfies
\begin{equation}
\label{eq:lc}\lela\nabla^g_xy,z\rira=\frac12\left(\lela  [x,y],z\rira-\lela[x,z],y\rira-\lela[y,z],x)\right),\qquad \forall\; x,y,z\in \mg.
\end{equation} Here and henceforth $\so(\mg)$ is the space of skew-symmetric endomorphisms of $(\mg,g)$.

We are interested in the conformal geometry of $(G,c:=[g])$, where $[g]$ denotes the conformal class of $g$. Any differential 1-form $\theta$ on $G$ defines a Weyl connection $D$ on $(G,c)$ by the fundamental theorem of conformal geometry \cite{We23}, and this correspondence is one-to-one. Under this correspondence, $\theta$ is called the Lee form of $D$; the Weyl structure is called closed (exact) if its Lee form is closed (resp.~exact). We say that a Weyl structure $D$ on $(G,c)$ is left-invariant if its Lee form with respect to the left-invariant metric $g$ is itself left-invariant on $G$; in this case, we also denote its value at the identity by $\theta\in \mg^*$.

Every left-invariant Weyl connection $D$ can be written in terms of its Lee form $\theta\in\mg^*$, the metric $g$ and the Levi-Civita connection $\nabla^g$ as $D=\nabla^\theta$ where 
\begin{equation}
\label{eq:weylc}
\nabla^\theta_xy:=\nabla^g_xy+\theta(x)y+\theta(y)x-g(x,y)\ts,\qquad \forall\; x,y\in \mg.
\end{equation}
Here and henceforth $\ts\in \mg$ denotes the $g$-dual vector of $\theta$. Using \eqref{eq:lc} together with \eqref{eq:weylc}, the Weyl connection $\nabla^\theta$ can be explicitly defined by the following formula:
    \begin{multline}
\label{eq:LCP}
\lela\nabla^\theta_xy,z\rira=\frac12\left(\lela  [x,y],z\rira-\lela[x,z],y\rira-\lela[y,z],x)\right)\\
+\theta(x)\lela y,z\rira+\theta(y)\lela x,z\rira-\theta(z)g(x,y),\qquad \forall\; x,y,z\in \mg.
\end{multline}
Notice that $\nabla^\theta$ can be viewed as a linear map $\nabla^\theta:\mg\to \End(\mg)$, $x\mapsto \nabla_x^\theta$. 

The curvature tensor $R^\theta$ of $\nabla^\theta$ is defined as
\[
R_{x,y}^\theta=[\nabla_x^\theta,\nabla_y^\theta]-\nabla_{[x,y]}^\theta, \qquad \forall\; x,y\in\mg.
\]

A subspace $\mmu\subset \mg$ is called $\nabla^\theta$-parallel if 
\[
\nabla_x^\theta u\in \mmu, \qquad \forall\; u\in\mmu,\ \forall\; x\in \mg, 
\]and it is called $\nabla^\theta$-flat if it is $\nabla^\theta$-parallel and 
\[
R_{x,y}^\theta u=0, \qquad \forall\; u\in\mmu,\ \forall\; x,y\in \mg.
\]

We are now ready to introduce the objects of study of this paper. 

\begin{defi} \label{defi} A locally conformally product (LCP for short) Lie algebra is a quadruple $(\mg,g,\theta,\mmu)$ where $(\mg,g)$ is a metric Lie algebra, $\theta$ is a non-zero closed $1$-form on $\mg^*$, and $\mmu$ is a $\nabla^\theta$-flat proper subspace of $(\mg,g)$, that is, $0\subsetneq \mmu\subsetneq \mg$. The LCP structure $(g,\theta,\mmu)$ is called {\em adapted} if $\theta|_\mmu=0$.
\end{defi}

In \cite[Theorem 4.2]{dBM24} it was shown that any LCP structure on a unimodular Lie algebra is adapted (the assumption that the LCP structure is proper in \cite[Theorem 4.2]{dBM24} is now part of Definition \ref{defi}, which no longer considers conformally flat Lie algebras as LCP Lie algebras). In addition, from \cite[Corollary 4.10]{dBM24} it follows that unimodular LCP Lie algebras are in one-to-one correspondence with triples $(\mh, h, \beta)$, where $\mh$ is a non-unimodular Lie algebra, $h$ is an inner product on $\mh$ and $\beta$ is an orthogonal Lie algebra representation $\beta:\mh\to \so(\R^q)$, with $q\ge 1$. We briefly recall here this correspondence.  

Given any such triple $(\mh, h, \beta)$, set $\xi:=-\frac1{q}H ^\mh$, where $H^\mh\in \mh^* $ is the trace form $H^\mh(x):=\tr\ad_x$, for all $x\in\mh$. Moreover, define the representation $\alpha:\mh\to \gl(\R^q)$ as $\alpha(x):=\xi(x)\Id_{\R^q}+\beta(x)$, for all $x\in\mh$. Then, the Lie algebra 
\begin{equation}\label{mg}
    \mg:=\R^q\rtimes_\alpha \mh 
\end{equation}
with a metric $g$ such that $\R^q\bot\mh$ and $g|_{\mh}=h$, is an LCP Lie algebra with $\theta$ being the extension by zero of $\xi$ and flat factor $\mmu:=\R^q$.

Conversely, if $(\mg,g,\theta,\mmu)$ is a unimodular LCP Lie algebra, then $\mmu$ is an abelian ideal of $\mg$ and $\mh:=\mmu^\bot$ is a subalgebra so $\mg=\R^q\rtimes \mh$ and this decomposition is orthogonal. Moreover, for every $x\in\mh$, $\beta(x):=\ad_x|_\mmu-\theta(x)\Id_\mmu\in \so(\mmu)$ and one can easily check that $\beta:\mh\to \so(\mmu)$ is a Lie algebra representation.

The following result allows us to define a maximal flat factor of an LCP structure on a unimodular Lie algebra, when the metric and the Lee form are fixed.

\begin{pro}\label{pro:max}Let $(\mg,g)$ be a unimodular metric Lie algebra and let $\theta\in \mg^*$ be a non-zero closed form such that $\nabla^\theta$ is not flat on $\mg$. Then, there exists a unique abelian ideal $\mmu$ in $\mg$ such that any LCP structure $(g,\theta,\mmi)$ on $\mg$  satisfies $\mmi\subset \mmu$. 
\end{pro}

\begin{proof}
If $(\mg,g)$ does not admit any LCP structure with Lee form $\theta$, it is clear that $\mmu=0$ satisfies the requirements. In the case where $(\mg,g)$ admits an LCP structure, we define $\mmu$ as the sum of all $\nabla^\theta$-flat subspaces of $\mg$. Clearly $\mmu$ is a $\nabla^\theta$-flat space, which is strictly contained in $\mg$ by the assumption that $\nabla^\theta$ is not flat. Thus $(\mg,\theta,\mmu)$ is an LCP Lie algebra, so \cite[Theorem 4.9]{dBM24} and \cite[Proposition 3.2]{AdBM} imply that $\mmu$ is an abelian ideal.
\end{proof}

\begin{defi}
Let $\mg$ ba a unimodular Lie algebra. We say that an LCP structure $(g,\theta,\mmu)$ on $\mg$ has maximal flat factor if $\nabla^\theta$ non-flat, and $\mmu$ is the ideal given by Proposition \ref{pro:max}.
\end{defi}

The above results are motivated by the more general notion of LCP structures on Riemannian manifolds which we will now discuss.

\subsection{LCP structures on compact manifolds}\label{LCPgen}

Recall that an LCP structure on a compact manifold $M$ is a pair $(g,\theta)$ consisting on a Riemannian metric $g$ and a closed 1-form $\theta$ which is not exact, and such that the Weyl connection $\nabla^\theta$ whose Lee form with respect to $g$ is equal to $\theta$  has reducible non-trivial holonomy (see \cite[Definition 2.1]{AdBM}). 

The pull-backs of $\theta$ and $g$ to the universal cover $\tilde M$ of $M$ will be denoted by the same symbols for simplicity. On $\tilde M$ we can write $\theta=\d\varphi$ for some $\varphi\in C^\infty(\tilde M)$ which is not $\pi_1(M)$-invariant (as $\theta$ is not exact on $M$). The lift of $\nabla^\theta$ to $\tilde M$ is the Levi-Civita connection of the metric $h:=e^{2\varphi}g$, which is therefore reducible and not flat. A fundamental result of Kourganoff is the following:

\begin{teo}[{\cite[Theorem 1.5]{Kourganoff}}] \label{teo:ko} The Riemannian manifold $(\tilde M,h)$ has a de Rham decomposition
\begin{equation}\label{ko}
    (\tilde M,h) \simeq \R^q \times (N,g_N), 
\end{equation}
where $\R^q$ is a flat Euclidean space of dimension $q\geq 1$ and $(N,g_N)$ is an irreducible and incomplete Riemannian manifold of dimension at least $2$. 
\end{teo}

The group $\Gamma:=\pi_1(M)$ acts on $\tilde M$ isometrically with respect to $g$ and by similarities with respect to $h$. Indeed, every $f\in\Gamma$ preserves $\d\varphi=\theta$, so $\d(f^*\varphi-\varphi)=0$. This shows that there exists a constant $c_f$ such that $f^*\varphi-\varphi=c_f$, whence $f^*h=e^{2c_f}h$.
Since similarities preserve the de Rham decomposition \eqref{ko}, we have $\Gamma\subset{\rm Sim}(\R^q)\times{\rm Sim}(N, g_N)$. Correspondingly, we denote each $f\in\Gamma$ by $(f_1,f_2).$

Following Kourganoff, we consider the group $\rmP:=\pr_{{\rm Sim}(N, g_N)}(\Gamma)$, in other words
\begin{equation}\label{eq:Pdefgen}
\rmP=\{f_2\in {\rm Sim}(N, g_N)\colon \exists f_1\in {\rm Sim}(\R^q) \mbox{ such that } (f_1,f_2)\in \Gamma\}.
\end{equation}

\begin{lm}\label{lm:2.5}
The projection $\pr_{{\rm Sim}(N, g_N)}$ defines an isomorphism between $\Gamma$ and $\rmP$.
\end{lm}
\begin{proof} This fact is a consequence of the proof of Lemma 4.17 in \cite{Kourganoff}. Notice that the proof of that lemma is incomplete, but it is sufficient to show our claim by taking $u''$ there as the identity on $N$. \end{proof}

We denote by $\bar \rmP$ the closure of $\rmP$ in ${\rm Sim}(N,g_N)$ and by $\bar \rmP_0$ its connected component of the identity. The following result was already used by Kourganoff \cite{Kourganoff}. We give here the proof for the reader's convenience.

\begin{lm}\label{5.2} Let $\rmP$ be a subgroup of a  Lie subgroup of a Lie group $G$ and let $\bar \rmP_0$ denote the connected component of the identity of the closure of $\rmP$ in $G$. Then for every sequence $(p_n)_{n\ge 0}$ in $\rmP$ converging to $p\in\bar \rmP_0$, there exists $n_0\in\N$ such that $p_n\in \bar \rmP_0\cap \rmP$ for every $n\ge n_0$. That is, $\overline{\rmP \cap {\bar\rmP}_0}={\bar\rmP}_0$.
\end{lm}
\begin{proof} If the conclusion does not hold, there exists a subsequence $p_{k_n}$ in $\rmP\setminus \bar \rmP_0$ converging to  some $p\in \bar \rmP_0$. But $p_{k_n}\in \bar \rmP\setminus \bar \rmP_0$ which is closed in $\bar \rmP$ (since $\bar \rmP$ is a Lie group, and the connected component of the identity in a Lie group is open). This implies that $p=\lim_{n\to\infty}p_{k_n}\in \bar \rmP\setminus \bar \rmP_0$, which is a contradiction.
\end{proof}

Set $\Gamma_0:=\Gamma\cap ({\rm Sim}(\R^q)\times \bar\rmP_0)$. By \cite[Lemma 4.18]{Kourganoff}, $\Gamma_0$ is contained in $\R^q\times \bar\rmP_0$ where the inclusion $\R^q\subset {\rm Sim}(\R^q)$ is given by the translations. We recall the following result:

\begin{pro}\cite[Lemmas 4.1 and 4.13]{Kourganoff}\label{pro:Kp0g0gen}
The group $\bar \rmP_0$ is a connected abelian Lie subgroup in ${\rm Iso}(N,g_N)$ and $\Gamma_0$ is a lattice in $\R^q\times \bar\rmP_0$.
\end{pro}

Note that throughout this paper, by lattice we mean a discrete co-compact subgroup.

\begin{defi}
The {\em characteristic group} of an LCP structure $(g,\theta)$ on the compact manifold $M$ is the connected abelian Lie group $\R^q\times \bar\rmP_0\subset {\rm Iso}( \R^q)\times {\rm Iso}(N, g_N)$. The {\em reduced characteristic group} is the connected abelian Lie group $\bar\rmP_0\subset {\rm Iso}(N, g_N)$.
\end{defi}

A central question in the theory of LCP manifolds is the following:

\begin{question}[{\cite[Remark 5.8]{Fl25}}]\label{quiz} Given an LCP structure on a compact manifold, is its characteristic group simply connected?    
\end{question}

Note that if this question holds true, then one can describe the corresponding LCP structures in terms of some {\em admissible data} (see \cite[Thm. 4.4 and 4.10] {Fl25} for details). We will investigate this question in the particular case of Lie LCP manifolds, which we shall introduce next.

\medskip

\begin{defi} \label{lielcp}
Let $(M,g,\theta)$ be an LCP manifold. 
If $\tilde M=G$ is a simply connected Lie group, the pull-backs to $G$ of $g$ and $\theta$ are left-invariant, and $\pi_1(M)=:\Gamma$ is a lattice of $G$ acting by left multiplication, the structure $(g,\theta)$ on $M=\Gamma\bs G$ is called a {\em Lie LCP structure} (see \cite[Definition 4.8]{FM25}). 
\end{defi}

The interest of this class of LCP structures is that it can be characterized algebraically: 

\begin{pro}\label{pro:11}
    There is a one-to-one correspondence between Lie LCP structures on $M:=\Gamma\bs G$ and LCP structures with flat maximal factor on the Lie algebra $\mg$ of $G$.
\end{pro}
\begin{proof}
    In \cite[Proposition 2.4]{AdBM} we have shown that if $\Gamma$ is a lattice in a simply connected Lie group $G$ with Lie algebra $\mg$, then
there is a correspondence between LCP structures on $\mg$ and Lie LCP structures on $M:=\Gamma\bs G$. We recall here how this correspondence is defined, and show that it is one-to-one if one restricts to LCP structures on $\mg$ with flat maximal factor.

Every LCP structure $(g,\theta,\mmu)$ on $\mg$ determines in an obvious way a Lie LCP structure $(g,\theta)$ on $\Gamma\bs G$ such that the left-invariant distribution determined on $G$ by $\mmu$ is $\nabla^\theta$-flat, and thus contained in the flat distribution $T\R^q\subset T G$. In particular, this shows that 
\begin{equation}\label{uq}
    \dim(\mmu)\leq q.
\end{equation}

Conversely, a Lie LCP structure $(g,\theta)$ on $M=\Gamma\bs G$, determines in an obvious way a scalar product $g\in\Sym^2\mg^*$ and a non-zero closed 1-form $\theta\in\mg^*$. Moreover, the flat distribution $T\R^q$ is left-invariant on $G$ (as left translations are isometries, and every isometry preserves the flat distribution in the de Rham decomposition), so it induces a subspace $\mmu\subset \mg$ which is clearly $\nabla^\theta$-flat and satisfies $\dim(\mmu)=q$. Thus $(g,\theta,\mmu)$ is an LCP structure on $\mg$. This structure has maximal flat factor, since if $\mmu$ were not maximal, there would exist an LCP structure $(g,\theta,\tilde\mmu)$ on $\mg$ with $\dim\tilde \mmu>\dim(\mmu)=q,$ thus contradicting \eqref{uq}.
\end{proof}

\section{semi-direct products defined by conformal representations}

We have seen in \eqref{mg} that every LCP Lie algebra can be expressed in a canonical way as a semi-direct product of Lie algebras. Correspondingly, the universal cover $G$ of any Lie LCP manifold $M=\Gamma\bs G$ can be realized as a semi-direct product of Lie groups, defined by a conformal representation (see Definition \ref{conformal} below). We will therefore study in more detail such "conformal" direct products, and show that the conformal class of the natural left-invariant metric carries a Riemannian product metric. We then apply this result in order to make explicit the two factors in Kourganoff's decomposition \eqref{ko} of the universal cover of any Lie LCP manifold.

We start with a general setting. Let $U$ and $H$ be Lie groups with Lie algebras $\mmu$ and $\mh$, respectively. Assume that $g_\mmu$ and $g_\mh$ are inner products on $\mmu$ and $\mh$, respectively, which will also denote the induced left-invariant metrics on the corresponding Lie group.

Given a Lie group representation $\rho:H\to \Aut(U)$, we denote $\alpha:\mh\to\Aut(\mmu)$ the induced Lie algebra representation.

We consider the semi-direct product of these Lie groups via $\rho$, $G=U\rtimes_\rho H$, whose product law is given by
\begin{equation}\label{eq:psd}
(u,x)(u',x')=(u\rho(x)u',xx'), \qquad \forall\; u,u'\in U,\ \forall\; x,x'\in H.
\end{equation}
In general, we will denote $ux$ the pair $(u,x)\in G$. Notice that, with this notation,  $xux^{-1}=\rho(x)u$ for all $u\in U$, $x\in H$.

Recall that the Lie algebra of $G$ is $\mg=\mmu\rtimes_\alpha \mh$. Let $g$ denote the left invariant metric on $G$ induced by $g_\mmu+g_\mh$ on $\mg$. 
\begin{defi}\label{conformal}
We say that a representation $\rho:H \to \Aut(U)$ is conformal if 
\begin{equation}\label{eq:cr}
\rho(x)^*g_\mmu=e^{2\varphi(x)}g_\mmu, \qquad \forall\; x\in H,
\end{equation}for some positive function $\varphi\in C^\infty(H)$, which is called the conformal function of $\rho$.
\end{defi}
\begin{remark}\label{rem:gu}
Notice that since $g_\mmu$ is left-invariant in $U$, $\rho$ being conformal is equivalent to $\d(\rho(x))_e^*g_\mmu=e^{2\varphi(x)}g_\mmu$ for all $x\in H$, where, in the last equality,  $g_\mmu$ is the inner product in $\mmu$.
\end{remark}
 
Using the fact that $\rho$ is a representation, it is straightforward to check that the conformal function $\varphi$ in \eqref{eq:cr} is a group homomorphism from $H$ to $(\R, +)$.

\begin{pro}\label{pro:sdisom}
If $\rho:H\to \Aut(U)$ is a conformal representation with conformal function $\varphi$, then the map $f:U\times H\to G$, defined by $f(u,x)=ux$, is an isometry between the Riemannian manifolds $(U\times H,h':=g_\mmu +e^{2\varphi}g_\mh)$ and $(G,h:=e^{2\varphi}g)$.
\end{pro}
Notice that if the conformal function $\varphi$ is non-constant, then the metrics $h$ and $h'$ are not left-invariant. 
\begin{proof}
For every $V\in\mmu$, $X\in \mh$, $u\in U$, and $x\in H$, we have:
\begin{eqnarray*}
\d f_{(u,x)}(\d (L_u)_e V, 0)&=& \frac{d}{dt}|_{t=0}f(u\exp tV,x) =\frac{d}{dt}|_{t=0}ux\rho(x^{-1})(\exp tV)\\&=&\d (L_{ux})_e\d(\rho(x^{-1}))_eV\\
\d f_{(u,x)}(0,\d (L_x)_e X)&=& \frac{d}{dt}|_{t=0}f(u,x\exp tX) =\frac{d}{dt}|_{t=0}ux\exp tX=\d (L_{ux})_eX.
\end{eqnarray*}

Using these relations together with \eqref{eq:cr} we get:
\begin{eqnarray*}
(f^*h)_{(u,x)}((\d (L_{u})_e V,0),(\d (L_{u})_e V,0))&=&h_{ux}(\d (L_{ux})_e\d(\rho(x^{-1}))_eV,\d (L_{ux})_e\d(\rho(x^{-1}))_eV)\\
&=&e^{2\varphi(x)}g_e(\d(\rho(x^{-1}))_eV,\d(\rho(x^{-1}))_eV)\\
&=&e^{2\varphi(x)}e^{2\varphi(x^{-1})}g_e(V,V)\\
&=&g_\mmu(V,V)=(g_\mmu)_u(d(L_u)_eV,d(L_u)_eV)\\
&=&h'_{(u,x)}
((\d (L_{u})_e V,0),(\d (L_{u})_e V,0))
.
\end{eqnarray*}

Similarly, one has:
\begin{eqnarray*}
(f^*h)_{(u,x)}((0,\d (L_{x})_e X),(0,\d (L_{x})_e X))&=&h_{ux}(\d (L_{ux})_eX,\d (L_{ux})_eX)\\
&=&e^{2\varphi(x)}g(X,X)\\
&=&e^{2\varphi(x)}(g_\mh)_{x}(\d (L_{x})_e X,\d (L_{x})_e X)\\
&=&h'_{(u,x)}((0,\d (L_{x})_e X),(0,\d (L_{x})_e X)).
\end{eqnarray*}

Finally, we compute:
\begin{eqnarray*}
(f^*h)_{(u,x)}((\d (L_{u})_e V,0),(0,\d (L_{x})_e X))&=&h_{ux}(\d (L_{ux})_e\d(\rho(x^{-1}))_eV,\d (L_{ux})_eX)\\
&=&e^{2\varphi(x)}g(\d(\rho(x^{-1}))_eV,X)=0\\
&=&h'_{(u,x)}((\d (L_{u})_e V,0),(0,\d (L_{x})_e X)).
\end{eqnarray*}Therefore $f^*h=h'$ as claimed. 
\end{proof}

\begin{lm}\label{lm:phihom} Assume $\varphi:G\to \R$ is a differentiable function such that $\varphi(e)=0$ and $\d \varphi$ is a left-invariant 1-form.  Then, $\varphi$ is a Lie group morphism into $(\R,+)$. 
\end{lm}

\begin{proof}
The left-invariance of $\d\varphi$ is equivalent to $L_a^*\d \varphi=\d \varphi$ for all $a\in G$. Since $G$ is connected, we obtain, for all $a\in G$, $L_a^*\varphi-\varphi=c_a$ where $c_a$ is a constant depending on $a$. This implies $\varphi(aa')-\varphi(a')=c_a$ for all $a,a'\in G$. Taking $a'=e$ we get $\varphi(a)=c_a$ and thus $\varphi(aa')=\varphi(a)+\varphi(a')$ as claimed. \end{proof}

Let $G$ be a Lie group with a left-invariant Riemannian metric $g$ and let $\varphi:G\to \R$ be a Lie group morphism. Then, for any $a\in G$,
\begin{equation}\label{eq:Laeg}
L_a^*(e^{2\varphi}g)=e^{2\varphi(a)}(e^{2\varphi}g).
\end{equation}
This means that $L_a$ is a similarity of $e^{2\varphi}g$ for all $a\in G$. Moreover, $L_a$ is an isometry of $e^{2\varphi}g$ if and only if $a\in\ker \varphi$.

\begin{pro}Suppose $\Gamma$ is a closed Lie subgroup of $G$ such that $\Gamma\bs G$ is compact and let $\varphi:G\to \R$ be a Lie group morphism. If the action of $\Gamma$ on $G$ by left-translations preserves the metric $e^{2\varphi}g$ then $\varphi\equiv 0$.
\end{pro}
\begin{proof} By \eqref{eq:Laeg}, it clear that $L_\gamma$ is an isometry of $e^{2\varphi}g$ for all $\gamma$ if and only if $\varphi(\Gamma)=0$. If that is the case, then $\Gamma\subset \ker \varphi$ and thus there is a well defined map $\Gamma\bs G\to \R$ from a compact space to $\R$ whose image is $\Im \varphi$. Since $\varphi$ is a group morphism, its image is compact only when it is trivial.
\end{proof} 

We will now apply the considerations above to the particular case of Lie LCP structures.

Let $(g,\theta)$ be a Lie LCP structure on $M:=\Gamma \bs G$, and let $(\mg, g,\theta,\mmu)$ be the corresponding LCP Lie algebra with maximal factor, given by Proposition \ref{pro:11}. The Lie algebra $\mg$ of $G$ decomposes as $\mg=\mmu\rtimes_\alpha \mh$, where $\mh=\mmu^\bot$ and $\alpha(x)=\theta(x)\Id_\mmu+\beta(x)$, for some representation $\beta:\mh\to\so(\mmu)$ (see Section \ref{sec:LCPLieAlg}). Recall that $\theta|_\mmu=0$ so we can assume $\theta\in\mh^*$.

Consider $U$ and $H$ the simply connected Lie groups corresponding to $\mmu$ and $\mh$,  and let $g_\mmu$, $g_\mh$ be the left-invariant metrics induced by $g$ on $U$ and $H$, respectively. 

Since $H$ and $U$ are simply connected, the representation $\alpha:\mh\to \gl(\mmu)$ induces Lie group homomorphisms $\rho_0:H\to \GL(\mmu)$ and $\rho:H\to \Aut(U)$. These satisfy the following relations: $\d(\rho_0)_e=\alpha$ and $\d(\rho(x))_e=\rho_0(x)$ for every $x\in H$. Notice that $G$ is (isomorphic to) the semi-direct product $U\rtimes_\rho H$. 

Since $\theta$ is a closed 1-form in $H$, which is simply connected, one can write $\theta=\d\varphi$ for some function $\varphi\in C^\infty(H)$ with $\varphi(e)=0$. By Lemma \ref{lm:phihom}, $\varphi: H\to (\R,+)$ is a group homomorphism.

\begin{lm}\label{lm:rhoisconf} In the notation above, $\rho$ is a conformal representation with conformal function $\varphi$.
\end{lm}
\begin{proof}Since $H$ is generated by products of elements in the image of the exponential map and $\rho$ and $\varphi$ in \eqref{eq:cr} are homomorphisms, it is enough to prove that $\rho$ is conformal for elements in $\exp(\mh)$.

Let $X\in \mh$. Then, 
\[
\rho_0(\exp X)=e^{\alpha(X)}=e^{\theta(X)}e^{\beta(X)},
\]where we used that $\d(\rho_0)_e=\alpha=\theta \Id_\mmu +\beta$. Since $\beta(X)\in\so(\mmu,g_\mmu)$, $e^{\beta(X)}\in \SO(\mmu,g_\mmu)$ and thus
\[
\rho_0(\exp X)^*g_\mmu=e^{2\theta(X)}g_\mmu.
\]

Using that $\varphi:H\to \R$ is the primitive of $\theta$ which is an homomorphism, we obtain $\varphi(\exp X)=\theta (X)$ so the previous equation reads
\[
\rho_0(\exp X)^*g_\mmu=e^{2\varphi(\exp X)}g_\mmu,
\]
thus proving $\d (\rho(x))_e^* g_\mmu=
\rho_0(x)^*g_\mmu=e^{2\varphi(x)}g_\mmu$ for all $x\in H$. 
The lemma follows from Remark \ref{rem:gu}.
\end{proof}

From Lemma \ref{lm:rhoisconf} and Proposition \ref{pro:sdisom} we obtain that $(G,e^{2\varphi}g)$ is isometric to the product $(U, g_\mmu)\times (H,e^{2\varphi}g_\mh)$. 
The next result shows that this is actually the de Rham decomposition of 
$(G,e^{2\varphi}g)$. 
 
\begin{pro}\label{pro:dRdec}In the notation above, the de Rham decomposition of the universal cover $G$ of the Lie LCP manifold $(M=\Gamma\bs G,g)$ endowed with the metric $e^{2\varphi}g$ reads
\begin{equation}\label{deco}
(G,e^{2\varphi}g)\simeq (U, g_\mmu)\times (H,e^{2\varphi}g_\mh).
\end{equation}
\end{pro}
\begin{proof}
 By Theorem \ref{teo:ko}, the de Rham decomposition of $(G,e^{2\varphi}g)$ has only two factors, one of them being flat Euclidean of dimension $q\geq1$.

From Proposition \ref{pro:11} and the fact that $(\mg,g,\theta,\mmu)$ is maximal, we get $q=\dim U$ and thus $\R^q\simeq (U,g_\mmu)$. Consequently, $(H,e^{2\varphi}g_\mh)$ is the irreducible factor given by Theorem \ref{teo:ko}.
\end{proof}

\section{The characteristic group of Lie LCP structures}

We keep the notation from the previous section: $(g,\theta)$ is a  Lie LCP structure on $M:=\Gamma \bs G$, $(\mg, g,\theta,\mmu)$ is the corresponding LCP Lie algebra with maximal factor, and $\mg=\mmu\rtimes_\alpha \mh$ is a semi-direct product defined by a conformal representation. The simply connected Lie groups corresponding to $\mmu$ and $\mh$ are denoted by $U$ and $H$, and  $g_\mmu$, $g_\mh$ denote the left-invariant metrics induced by $g$ on $U$ and $H$, respectively. We denote by $\varphi$ the homomorphism from $G$ to $(\R,+)$ that satisfies $\d\varphi=\theta$ given by Lemma \ref{lm:phihom}.

We identify $U$ with the leaf of $\mmu$ through the identity of $G$ (which is a normal abelian subgroup of $G$) and write as before $G=U\rtimes_\rho H$. 

\begin{remark}
[Projections in semi-direct products: notation]
Assume that $A$ is a group and $X,B\subset A$ are subgroups such that $X$ is normal in $A$, $A=XB$  and $X\cap B={1_A}$. Then $A$ can be written as the semi-direct product $A=X\rtimes B$, and there is a well-defined group morphism $\pr^A_B:A\to B$ such that $\pr^A_B(xb)=b$ for every $x\in X$ and $b\in B$. If $B$ itself is a semi-direct product $B=Y\rtimes C$, then one can write
$$A=X\rtimes (Y\rtimes C)=(X\rtimes Y)\rtimes C,$$
and clearly 
\begin{equation}\label{abc}
    \pr^A_C=\pr^B_C\circ\pr^A_B.
\end{equation}    
\end{remark}

Since $M=\Gamma\bs G$, the fundamental group $\pi_1(M)$ can be identified to the subgroup $\Gamma$ of $G$, acting on $G$ by left translations. However, we will make the distinction between an element $\gamma\in\Gamma$ and its action $L_\gamma$ on $G=\tilde M$.

We denote by $p_U:G\to U$ the projection over the first factor with respect to the identification of $U\times H$ with $G$ given by $(u,x)\mapsto ux$. Note that, unlike $\pr_H^G$,  $p_U$ is not a group morphism in general.  
Given $\gamma\in \Gamma$, write $\gamma=(\gamma_1, \gamma_2)$, with $\gamma_1:=p_U(\gamma)$ and $\gamma_2:=\pr^G_H(\gamma)$, so that for every  $a=(u,x)\in G$ with $u\in U$, $x\in H$, we have, by \eqref{eq:psd},
\begin{equation}
\label{eq:gact}
\gamma a=(\gamma_1\rho(\gamma_2)u, \gamma_2x).
\end{equation} 

We will now specialize the considerations about the characteristic group introduced in \S\ref{LCPgen} to the present setting of Lie LCP structures. As before, $\Gamma$ acts on $G=U\rtimes_\rho H$ by homotheties of $e^{2\varphi}g$. Therefore, it preserves the de Rham decomposition $(G,e^{2\varphi}g)=(U,g_\mmu)\times (H, e^{2\varphi}g_\mh)$ given by Proposition \ref{pro:dRdec} and thus $\Gamma\subset {\rm Sim} (U,g_\mmu)\times {\rm Sim}(H, e^{2\varphi}g_\mh)$. By denoting $f_1\in {\rm Sim} (U,g_\mmu)$ and $f_2\in {\rm Sim} (H,e^{2\varphi}g_\mh)$ the components of $L_\gamma\in \Gamma$, \eqref{eq:gact} yields $f_1(u)=\gamma_1\rho(\gamma_2)u$ and $f_2(h)=\gamma_2 h$, for every $u\in U$ and $h\in H$. In particular this shows that each element $\gamma\in \Gamma$ acts on $H$ as the left-translation by $\gamma_2=\pr^G_H(\gamma)$.

In the Lie LCP setting, the characteristic group reads
\begin{equation}\label{eq:Pdef}
\rmP=\{f_2\in {\rm Sim}(H, e^{2\varphi}g_\mh)\colon \exists f_1\in {\rm Sim}(U,g_\mmu) \mbox{ such that } (f_1,f_2)\in \Gamma\}.
\end{equation}
In other words, $\rmP=\pr^G_H(\Gamma)$ and by the above observation, $\rmP$ acts on $H$ by left-translations. By Lemma \ref{lm:2.5}, the projection $\pr^G_H$ defines an isomorphism between $\Gamma$ and $\rmP$.

Since the inclusion $H\subset {\rm Sim}(H, e^{2\varphi}g_\mh)$ given by left-translations is closed and continuous, the closure of $\rmP$ in ${\rm Sim}(H, e^{2\varphi}g_\mh)$ coincides with its closure in $H$. 
We denote as before by $\bar \rmP$ this closure and by $\bar \rmP_0$ its connected component of the identity. In particular, by Proposition \ref{pro:Kp0g0gen}, $\bar\rmP_0\subset H$ acts on $H$ by isometries of $e^{2\varphi}g_\mh$. From this we obtain: 
\begin{pro} \label{pro:vanish} The Lee form $\theta$ vanishes on the Lie algebra $\mmp$ of $\bar\rmP_0$. \end{pro}
\begin{proof} Since $\bar\rmP_0\subset H$ acts on $H$ by left-translations which are isometries of $e^{2\varphi}g_\mh$, Equation \eqref{eq:Laeg} applied to $(H,e^{2\varphi}g_\mh)$ shows that $\varphi(x)=0$ for any $x\in \bar\rmP_0$. Then for any $X\in \mmp$, $\theta(X)=\frac{d}{dt}|_{t=0}\varphi(\exp tX)=0$ as claimed. 
\end{proof}

A priori, we cannot realize the characteristic group $\R^q\times \bar\rmP_0$ (see Definition \ref{eq:Pdefgen}) of a Lie LCP manifold $M=\Gamma\backslash G$ as a subgroup of $G=U\rtimes_\rho H$, since the elements of $U\simeq \R^q$ and $\bar\rmP_0$ may not commute in $G$. However, the next result shows that this is always the case.

\begin{pro}\label{pro:xuux}
For all $u\in U$, $x\in \bar\rmP_0$, $ux=xu$ in $G=U\rtimes_\rho H$.
\end{pro}
\begin{proof}This is equivalent to showing  $\rho(x)=\Id_U$ for all $x\in \bar\rmP_0$.

Let $\gamma_2\in \bar\rmP_0\cap \rmP$. By the definition of $ \rmP$, there exists $\gamma_1\in {\rm Sim}(U,g_\mmu)$ with $\gamma=(\gamma_1,\gamma_2)\in \Gamma$. Thus $\gamma\in \Gamma \cap ({\rm Sim}(U,g_\mmu)\times \bar\rmP_0)=\Gamma_0$ and, by \cite[Lemma 4.18]{Kourganoff}, the restriction of $\gamma$ to $U$ is a translation. Hence, there exists $w\in U$ such that 
\[
\gamma_1\rho(\gamma_2)u=wu, \qquad \forall\; u\in U.
\]

Taking $u$ as the identity in $U$ in the last equation yields $\gamma_1=w$ and thus $\rho(\gamma_2)=\Id_U$ for all $\gamma_2\in \bar\rmP_0\cap \rmP$.
Since $\bar\rmP_0\cap \rmP$ is dense in $\bar\rmP_0$ (Lemma \ref{5.2}), by continuity we get $\rho(x)=\Id_U$ for all $x\in \bar\rmP_0$.
\end{proof}

\begin{cor}\label{cor4.3}The characteristic group $\R^q\times \bar\rmP_0$ of a Lie LCP manifold $M=\Gamma\bs G$ is an abelian subgroup of $G=U\rtimes_\rho H$.
\end{cor}

Recall that $\Gamma_0=\Gamma\cap (U\times \bar\rmP_0)$ by Proposition \ref{pro:Kp0g0gen}.

\begin{pro}\label{4.8}
$\Gamma_0$ is a normal subgroup of $\Gamma$.
\end{pro}
\begin{proof} Since $\Gamma_0$ is a subgroup of $U\times \bar\rmP_0$ and $U$ is normal in $G$, by \eqref{eq:psd} it suffices to show that $ \bar\rmP_0$ is preserved by conjugation with elements in $P$, which is obvious since $ \bar\rmP_0$ is normal in $ \bar\rmP$.
\end{proof}

\begin{lm}\label{lm:pip}
$\rmP\cap \bar\rmP_0=\pr^G_H(\Gamma_0)$.
\end{lm}
\begin{proof}Let $\gamma_2\in  \rmP\cap \bar\rmP_0$, then there is $\gamma_1\in U$ such that $\gamma:=(\gamma_1,\gamma_2)\in\Gamma$. Since $\gamma_2\in \bar\rmP_0$ we get $\gamma\in U\times \bar\rmP_0$ as well and thus $\gamma\in\Gamma_0$, from which we obtain $\gamma_2\in \pr^G_H(\Gamma_0)$.

Conversely, let $\gamma_2\in\pr^G_H(\Gamma_0)$. Take $\gamma_1\in U$ such that $\gamma:=(\gamma_1,\gamma_2)\in \Gamma_0=\Gamma\cap (U\times \bar\rmP_0)$. Then $\gamma_2\in \pr^G_H(\gamma)=\rmP$ and $\gamma_2\in \pr^G_H(U\times \bar\rmP_0)=\bar\rmP_0$.
\end{proof}

\begin{pro} \label{pro:finsg}Assume that $\Gamma$ admits a finite index subgroup $\Lambda$. Denote by $Q:=\pr^G_H(\Lambda)$.
Then $\Lambda$ is a lattice of $G$ and $(g,\theta)$ is a Lie LCP structure on $\Lambda\bs G$ whose reduced characteristic group $\bar{\rm Q}_0$ verifies $\bar{\rm Q}_0=\bar\rmP_0$.
\end{pro}
\begin{proof}
The fact that $\Lambda$ is a subgroup of $\Gamma$ implies that $\Lambda$ is discrete and the map from $\Lambda\bs G$ to $\Gamma\bs G$ defined as $\Lambda x\to \Gamma x$ is a well defined finite covering of a compact space. Hence, $\Lambda\bs G$ is compact.

Let $\ell_1, \ldots, \ell_k$ be elements in $\Gamma$ such that $\Gamma=\bigcup_{i=1}^k \Lambda\, \ell_i$ and denote $p_i:=\pr^G_H(\ell_i)\in \rmP$.  Let $x\in \bar\rmP$, so that $x=\lim_nx_n$ with $x_n\in \rmP=\pr^G_H(\Gamma)$. Up to passing to a subsequence, we may assume that there is some $i_0\in \{1, \ldots, k\}$ for which $x_n= y_n p_{i_0}$,  with $y_n\in {\rm Q}$ for all $n\in\N$. Hence $y_n$ converges to some $y\in \bar {\rm Q}$, whence $x\in \bar{\rm Q}\, p_{i_0} $. This implies that 
\[
\bar\rmP=\bigcup_{i=1}^k\bar {\rm Q}\,p_{i}.
\]In this union, two sets either intersect or are disjoint since $\bar{\rm Q}$ is a subgroup of $\bar\rmP$. Therefore,  $\bar{\rm Q}_0=\bar\rmP_0$.
\end{proof}

To finish the section we present specific results on Lie LCP {\em solvmanifolds}, that is, on manifolds of the form $M=\Gamma\bs R$, where $R$ is a simply connected solvable Lie group and $\Gamma\subset R$ a lattice. 

First, we answer positively Question \ref{quiz} for LCP structures on solvmanifolds:

\begin{pro}\label{pro:charsol} On a solvmanifold $M=\Gamma\bs R$, every Lie LCP structure has simply connected characteristic group.
\end{pro}
\begin{proof}By Corollary \ref{cor4.3}, the characteristic group $U\times \bar\rmP_0$ of $M$ is a Lie subgroup of $R$. Since $R$ is solvable and simply connected,  $U\times \bar\rmP_0$ is simply connected as well \cite[Theorem 3.18.12]{VAR}.
\end{proof}

The last result of this section involves a finite index subgroup of the lattice $\Gamma_0$ of the characteristic group of LCP solvmanifolds. For its proof, we make use of the following technical result on lattices on vector spaces.

\begin{lm}\label{5.1} Let $E$ be a real vector space, $\Gamma_0\subset E$ a lattice, $E_1,E_2\subset E$ two vector subspaces such that $E=E_1\oplus E_2$ and $f$ an endomorphism of $E$ such that $f(\Gamma_0)\subset \Gamma_0$, $f(E_i)\subset E_i$ for $i=1,2$, and such that the restriction to $E_1$ of $f-\Id_E$ is invertible. Denote by $\pi_i:E\to E_i$ the projections with respect to the above direct sum decomposition and assume that $\pi_2(\Gamma_0)$ is dense in $E_2$. Then $f-\Id_E$ is invertible.
\end{lm}

\begin{proof} Let $k$ be the dimension of $E$. The lattice $\Gamma_0$ is a free abelian group of rank $k$. Let $e_1,\ldots,e_k$ be a system of generators of $\Gamma_0$ (and in the same time a basis of $E$). The matrix $A$ of $f$ with respect to this basis has integer coefficients. We endow $E$ with the scalar product such that $\{e_i\}_{i=1}^k$ is an orthonormal basis. 

Assume for a contradiction that $f-\Id_E$ is not invertible and let $x$ be a non-zero vector in $\ker(A^T-I_k)$. Since $A\in M_k(\Z)$, one can choose $x$ such that it has integer coefficients. By assumption, $E_1\subset \mathrm{im}(A-I_k)$, whence by duality, $\ker(A^T-I_k)\subset E_1^\perp$. This shows that $x\in E_1^\perp$, so by duality again, $E_1\subset x^\perp$. Every vector $z\in\Gamma_0$ can be written $z=(z-\langle x,z\rangle \frac{x}{|x|^2})+\langle x,z\rangle\frac{x}{|x|^2}$, thus showing that $\Gamma_0$ is contained in $x^\perp+\frac{x}{|x|^2}\Z$.

Now, for every $y\in x^\perp$, we have $\pi_2(y)=y-\pi_1(y)\in x^\perp$, so $\pi_2(x^\perp)\subset x^\perp\cap E_2=:W$. This intersection is not equal to $E_2$ (otherwise, since we already have $E_1\subset x^\perp$, we would obtain $E\subset x^\perp$, which is impossible), so $W$ is a hyperplane in $E_2$. Consequently, $\pi_2(x^\perp)$ is contained in the hyperplane $W$, whence $\pi_2(\Gamma_0)$ is contained in the closed (strict) subset $W+\pi_2(\frac{x}{|x|^2})\Z$ of $E_2$, contradicting the fact that it has to be dense in $E_2$.
\end{proof}

\begin{pro} \label{pro:findsol} Let $M=\Gamma\bs R$ be a Lie LCP solvmanifold defined by a solvable LCP Lie algebra $(\mr,g,\theta,\mmu)$ and a lattice $\Gamma$ in the simply connected solvable Lie group $R$ with Lie algebra $\mr$. Denote as before by $\bar\rmP_0$ the reduced characteristic group of the LCP structure (which is simply connected by Proposition \ref{pro:charsol}) and by $\Gamma_0:=\Gamma\cap (U\times \bar \rmP_0)$ which is normal in $\Gamma$ by Proposition \ref{4.8} and a lattice in $U\times \bar \rmP_0$ by Proposition \ref{pro:Kp0g0gen}. Then the group $[\Gamma,\Gamma_0]$ of elements of $\Gamma_0$ generated by commutators of an element in $\Gamma$ and an element of $\Gamma_0$ has finite index in $\Gamma_0$. 
\end{pro}

\begin{proof} Let $\gamma$ be an element of $\Gamma\subset R$ which is a strict homothety of $(R,e^{2\varphi}g)$ (where $\varphi$ is a primitive of $\theta$ on $R$). Then $\gamma^*\varphi=\varphi+c$ for some non-zero constant $c$. 
We write $\gamma=(\gamma_1, \gamma_2)$, with $\gamma_1=p_U(\gamma)$ and $\gamma_2=\pr^G_H(\gamma)$. For every $u\in U$, the adjoint action of $\gamma$ on $a:=(u,1)\in R=U\rtimes_\rho H$ reads by \eqref{eq:psd}:
\begin{equation*}
\Ad_\gamma(a)=\Ad_{\gamma_1}(\Ad_{\gamma_2}(a))=\Ad_{\gamma_1}(\rho(\gamma_2)a)=\rho(\gamma_2)a.
\end{equation*}
Therefore, by \eqref{eq:cr}, the adjoint action of $\gamma$ on $(U,g_\mmu)$ is a strict homothety. Moreover, $\Ad_\gamma$ is a morphism of the characteristic group $U\times \bar \rmP_0$, which is abelian and simply connected, so has a vector space structure. We denote by $E$ this vector space. By continuity, $f:=\Ad_\gamma$ is a linear automorphism of $E$, which clearly preserves  the two vector subspaces $E_1:=U$ and $E_2:=\bar \rmP_0$, and verifies $f(\Gamma_0)\subset \Gamma_0$ by Proposition \ref{4.8}.

Moreover, if $\pi_2$ denotes the projection from $E$ to $E_2$ parallel to $E_1$, then $\pi_2(\Gamma_0)$ is dense in $E_2$ by Lemma \ref{5.2}. In fact $\rmP\cap {\bar\rmP}_0=\pi_2(\Gamma_0)$ due to $\Gamma_0=\Gamma\cap (U\times \bar\rmP_0)\subset U\times (\rmP\cap \bar\rmP_0)$ by Lemma \ref{5.2}. We can thus apply Lemma \ref{5.1} and obtain that $f-\Id_E$ is invertible. Consequently, the image of $\Gamma_0$ by $f-\Id_E$ is a finite index subgroup of $\Gamma_0$. On the other hand, the commutator of $\gamma$ and some $x\in \Gamma_0$ reads $[\gamma,x]=\gamma x\gamma^{-1}x^{-1}=\Ad_\gamma(x)-x=(f-\Id_E)(x)$.
This shows that $[\Gamma,\Gamma_0]$ contains the finite index subgroup $(f-\Id_E)(\Gamma_0)$ of $\Gamma_0$, thus finishing the proof.
\end{proof}

\section{Reduction to the radical of $G$}

 The aim of this section is to prove the main result of the paper:

\begin{teo}\label{main}
    The characteristic group of a Lie LCP manifold $M:=\Gamma \bs G$ is contained in the radical of $G$.
\end{teo}

Let $(g,\theta)$ be a Lie LCP structure on $M:=\Gamma \bs G$ and let $(\mg, g,\theta,\mmu)$ be the corresponding LCP Lie algebra with maximal factor.

As in the previous sections we consider $\mh:=\mmu^\bot$ so that $\mg=\mmu\rtimes_\alpha \mh$ and the corresponding simply connected Lie groups $U,H$ so that $G$ is isomorphic to $U\rtimes_\rho H$, where $\rho:H\to \Aut(U)$ is the conformal representation induced by $\alpha$ in \eqref{mg}.  By Proposition \ref{pro:dRdec}, the de Rham decomposition of $(G,e^{2\varphi}g)$ is $(U, g_\mmu)\times (H,e^{2\varphi}g_\mh)$, where $g_\mmu,g_\mh$ are the left-invariant metrics induced by the restriction of $g$ to $\mmu, \mh$, respectively, and $\varphi:G\to \R$ is the homomorphism verifying $\d\varphi=\theta$ given by Lemma \ref{lm:phihom}.

Fix a Levi decomposition of $H$ as follows. Let $R_0$ be the radical of $H$ (i.e. the maximal connected solvable Lie subgroup) and let $S$ be a semisimple Levi factor of $H$ so that $H=R_0S$. Note that $S$ is simply connected since $H$ is so \cite[Theorem 3.18.13]{VAR}. We denote by $\tau_0:S\to \Aut(R_0)$ the action of $S$ on $R_0$ by conjugation and write
\begin{equation}\label{eq:LeviH}
H=R_0\rtimes_{\tau_0} S.
\end{equation}

Since $U$ is abelian and normal in $G$, $R:=U\rtimes_\rho R_0$ is solvable and connected, and $G=RS$, hence 
\begin{equation}\label{eq:LeviG}
G=R\rtimes_\tau S= (U\rtimes_\rho R_0)\rtimes_\tau S
\end{equation}
is a Levi decomposition of $G$ where we denote by $\tau:S\to \Aut(R)$ the action of $S$ on $R$ by conjugation. Note that $\tau(s)|_U=\rho(s)$ and $\tau(s)|_{R_0}=\tau_0(s)$ for all $s\in S$.

The simply connected semisimple Lie group $S$ has a unique decomposition (up to order) as a product of simple Lie groups. Therefore, any connected normal subgroup is a product of some of these factors. In particular, the connected component of the kernel of $\tau$,  $(\ker \tau)_0\subset S$, can be written as such a product; we denote by $K$ the product of the compact factors of $(\ker \tau)_0$. Let $S_1$ be the product of simple and simply connected factors of $S$ which are not in $K$, so that
\begin{equation}\label{eq:SS}
S=S_1\times K.
\end{equation} 

Recall that the characteristic group $U\times \bar\rmP_0$ of the Lie LCP structure $(g,\theta)$ on $M=\Gamma\bs G$ is an abelian subgroup of $G$ (by Corollary \ref{cor4.3}). 

 Theorem \ref{main} is equivalent to showing that the reduced characteristic group $\bar\rmP_0$ is contained in $R_0$. We start by showing that this holds in some particular case.

\begin{pro}\label{pro:fker}
If  the connected component of the identity of the kernel of $\tau$ has no compact factor, then $\bar\rmP_0\subset R_0$.
\end{pro}
\begin{proof}The hypothesis is equivalent to $K$ in \eqref{eq:SS} being trivial, so one can write the different semi-direct product decompositions of $G$ as
\begin{equation}
G=(U\rtimes_\rho R_0)\rtimes_\tau S_1=U\rtimes_\rho (R_0\rtimes_{\tau_0} S_1),
\end{equation}
with 
$$R=U\rtimes_\rho R_0\qquad\mbox{and}\qquad H=R_0\rtimes_{\tau_0} S_1.$$
Since $(\ker\tau)_0$ has no compact factors, $\pr^G_{S_1}(\Gamma)$ is a lattice in $S_1=R\bs G$ 
by \cite[Corollary 8.28]{RA} (see also \cite[Proposition 1.3]{Wu88}). Using \eqref{abc}, this implies that $\pr^H_{S_1}(\rmP)=\pr^H_{S_1}\circ \pr^G_H(\Gamma)=
\pr^G_{S_1}(\Gamma)$ is also a  lattice in $S_1$.

Let $x\in \bar \rmP$ and let $x_n\in \rmP\subset H$ such that $x_n\to x$. Then $\pr^H_{S_1}(x_n)\in \pr^H_{S_1}(\rmP)$ converges to $\pr^H_{S_1}(x)\in \pr^H_{S_1}(\bar\rmP)\subset S_1$. Since $\pr^H_{S_1}(\rmP)$ a lattice in $S_1$, the latter sequence is stationary, i.e.~there exists $n_0\in \N$ with $\pr^H_{S_1}(x_n)=\pr^H_{S_1}(x)$ for every $n\ge n_0$. This implies $\pr^H_{S_1}(\bar\rmP)= \pr^H_{S_1}(\rmP)$ and thus $\pr^H_{S_1}(\bar\rmP_0)\subset  \pr^H_{S_1}(\rmP)$. As the former set is connected and the latter is a lattice, this shows that $\pr^G_H(\bar\rmP_0)$ is trivial and therefore $\bar\rmP_0\subset R_0$.
\end{proof}

We consider the group $L:=(U\rtimes_\rho R_0)\rtimes_\tau S_1$, so that $G=L\times K$.
We show next that the projection of $\Gamma$ in $L$ is a lattice. This follows from a general argument:
\begin{lm}\label{lm:gam1lat}
Let $\Gamma$ be a lattice in a direct product of Lie groups $G=L\times K$ with $K$ compact. Then the subgroup $\Gamma^L:=\pr^G_L(\Gamma)$ is a lattice in $L$. 
\end{lm}

\begin{proof}It is clear that $\Gamma^L$ is a subgroup of $L$. 
We first show that $\Gamma^L$ is discrete. Assume for a contradiction that there exists a sequence $\{\gamma^L_n\}_{n\in\N}\subset L$ such that $\gamma_n^L\to x\in L$ and $\gamma_n^L\ne x$ for every $n\in \N$. For each $n\in \N$, there exists $\gamma_n^K\in K$ so that $(\gamma_n^L,\gamma_n^K)\in\Gamma$. Since $K$ is compact, there exists a subsequence $\{\gamma^K_{n_k}\}_{k\in\N}$ which converges to some $y\in K$. Thus $(\gamma^L_{n_k},\gamma^K_{n_k})\to (x,y)$ as $k\to\infty$. However, $\Gamma$ is discrete in $G$, so this last sequence is stationary. In particular $\gamma_n^L=x$ for $n$ large enough, which contradicts our assumption. Thus $\Gamma^L$ is discrete in $L$.

To show that $\Gamma^L\bs L$ is compact, consider the continuous surjective map $\psi:L\times K\to \Gamma^L\bs L$, $(x,y)\mapsto\Gamma^L x$. For $(\gamma_1,\gamma_2)\in \Gamma$ and $(x,y)\in L\times K$, we can write $\psi((\gamma_1,\gamma_2)(x,y))=\psi(\gamma_1 x,\gamma_2 y)=\Gamma^L x=\psi(x,y)$. Hence $\psi$ induces a continuous surjective map $\bar\psi:\Gamma\bs G\to \Gamma^L\bs L$, so $\Gamma^L\bs L$ is compact. 
\end{proof}

Recall that the group $\rmP=\pr^G_H(\Gamma)$ is a subset of $H=R_0\rtimes_\tau (S_1\times K)$. According to this semi-direct product structure, elements in $H$ will be denoted by $(r,x,s)\in H$, where $r\in R_0$, $x\in S_1$, $s\in K$. We will now show that the elements of the reduced characteristic group $\bar\rmP_0$ have no component in $S_1$:
 
 \begin{pro}\label{4.14}
 The projection of $\bar\rmP_0$ on $S_1$ is trivial. 
 \end{pro}
  \begin{proof} 
We shall first prove that, in the notation above,  
$\pr^H_{S_1}(\bar\rmP)=\pr^H_{S_1}(\rmP)$ and this group is a lattice is $S_1$.

By Lemma \ref{lm:gam1lat}, $\Gamma^L=\pr^G_{L}(\Gamma)$ is a lattice in $L$, and since $S_1$ has no compact factors acting trivially on $R$ via the representation $\tau|_{S_1}:S_1\to \Aut(R)$, \cite[Corollary 8.28]{RA} implies that $\pr^L_{S_1}(\Gamma^L)$ is a lattice in $S_1$. 

On the other hand, by \eqref{abc}, we can express
$$\pr^H_{S_1}(\rmP)=\pr^H_{S_1}(\pr^G_H(\Gamma))=\pr^G_{S_1}(\Gamma)=\pr^L_{S_1}(\pr^G_{L}(\Gamma))=\pr^L_{S_1}(\Gamma^L).$$
Thus $\pr^H_{S_1}(\rmP)$ is a lattice in $S_1$, whence 
$$\pr^H_{S_1}(\bar\rmP)\subset \overline{\pr^H_{S_1}(\rmP)}=\pr^H_{S_1}(\rmP).$$
  
Now, since $\bar\rmP_0$ is connected, $\pr^H_{S_1}(\bar\rmP_0)$ is a connected subgroup of the lattice $\pr^H_{S_1}(\rmP)$, so it is trivial. 
\end{proof}

Using this result, we will now show that the characteristic group of $M$ is isomorphic to the one of a "reduced" Lie LCP manifold, obtained by suppressing the semisimple factor $S_1$ form $G$. Below, we will identify $R\times K$ (resp.~$R_0\times K$) with the subgroup $R\times\{1\}\times K$ (resp.~$R_0\times\{1\}\times  K$) of $G=(R\rtimes_\tau S_1) \times K$.

\begin{pro}\label{p=q} The subgroup  $\Sigma:=\Gamma\cap(R\times K)$ is a lattice in $R\times K$ and the compact manifold $\Sigma \bs  (R\times K)$ has a Lie LCP structure. Moreover, the Lie group $\bar\rmP_0$ associated to the LCP manifold $\Gamma\bs G$ coincides with the reduced characteristic group of the Lie LCP manifold $\Sigma \bs  (R\times K)$.
\end{pro}
\begin{proof}
The subgroup $\Sigma$ is clearly discrete. To prove co-compactness, we use the fact that $\Gamma^L=\pr^G_L(\Gamma)$ is a lattice in $L=R\times S_1$ by Lemma \ref{lm:gam1lat}. Since $\tau_{|_{S_1}}:S_1\to \Aut(R)$ has no compact factors in the kernel, $\Gamma^L\cap R$ is a lattice in $R$ by \cite[Corollary 8.28]{RA}. Thus there exists a compact set $C$ in $R$ such that for every $r\in R$ there exists an element $r'\in \Gamma^L\cap R$ with $r'r\in C$. Since $r'\in \Gamma^L$, there exists $s'\in K$ such that $(r',s')\in \Gamma$. Thus, for every $(r,s)\in R\times K$, there exists an element $(r',s')\in \Gamma\cap (R\times K)=\Sigma$ such that $(r',s')\cdot (r,s)$ is in the compact set $C\times K$. Therefore, $\Sigma\bs  (R\times K)$ is the image of $C\times K\subset R\times K$ by the quotient map, hence it is compact. Thus $\Sigma$ is a lattice of $R\times K$. 

To show that $\Sigma\bs  (R\times K)$ has a Lie LCP structure, it is enough, by Proposition \ref{pro:11}, to check that the Lie algebra of $R\times K$ has an LCP structure. This follows directly from \cite[Corollary 4.8]{dBM24}, provided that the restriction to the Lie algebra $\mr\oplus\mk$ of $R\times K$ of the Lee form $\theta$ of the LCP structure of $\mg$ is non-zero. However, the restriction of $\theta$ to the Lie algebra $\ms_1$ of $S_1$ vanishes (since $S_1$ is semisimple and $\theta$ is closed), and $\mg=\mr\oplus\mk\oplus\ms_1$ as vector spaces, showing that $\theta|_{\mr\oplus\mk}\ne 0$.

We denote by ${\rm Q}:=\pr^{\,R\times K}_{R_0\times K}(\Sigma)$. Then $\bar{\rm Q}_0$, the connected component of the identity of its closure $\bar{\rm Q}$ in $R_0\times K$, is by definition the reduced characteristic group of the Lie LCP manifold $\Sigma\bs  (R\times K)$. We will show that, under the identification of $R_0\times K$ with $R_0\times \{1\}\times K$, one has $\bar\rmP_0=\bar{\rm Q}_0$.

 Proposition \ref{4.14} implies $\bar\rmP_0\subset R_0\times K$ and thus
any element in $\bar\rmP_0$ has the form 
$(r,1,s)$ for some $r\in R_0$, $s\in K$, where $1$ is the identity in $S_1$. By Lemma \ref{5.2}, there exist sequences $(r_n)\subset R_0$, $(s_n)\subset K$ such that $(r_n,1,s_n)\in \rmP\cap \bar\rmP_0$ 
and $(r_n,1,s_n)\to (r,1,s)$ in $R_0\rtimes (S_1\times K)$. Since $\rmP=\pr^G_H(\Gamma)$, there exists a sequence $(u_n)\subset U$ such that $(u_n,r_n,1,s_n)\in \Gamma$ for all $n$. In particular, $(u_n,r_n,1,s_n)\in \Gamma\cap (R\times K)=\Sigma$ and $(r_n,1,s_n)\in {\rm Q}$. Therefore, $(r,1,s)\in \bar {\rm Q}$, because it is the limit of $(r_n,1,s_n)$. Hence $\bar\rmP_0\subset \bar {\rm Q}_0$.

Conversely, if  $(r,1,s)\in \bar {\rm Q}_0$, then $(r,1,s)=\lim_{n\to \infty}(r_n,1,s_n)$, with $(r_n,1,s_n)\in {\rm Q}=\pr^{\,R\times K}_{R_0\times K}(\Sigma)$. Thus there exists a sequence $(u_n)\subset U$ such that $(u_n,r_n,1,s_n)\in \Sigma=\Gamma\cap (R\times K)$, and consequently $(u_n,r_n,1,s_n)\in \Gamma$. Projecting to $H=R_0\rtimes_\tau (S_1\times K)$ we obtain $(r_n,1,s_n)\in \rmP$ and therefore $(r,1,s)\in \bar\rmP$. It follows that $\bar\rmP_0=\bar {\rm Q}_0$.
\end{proof}

Due to Proposition \ref{p=q}, in order to prove Theorem \ref{main} we can assume from now on that $S_1$ is trivial, i.e. $M$ is a Lie LCP manifold of the form $M=\Gamma\bs G$, where $G=R\times K$ with $R$ solvable, $K$ compact semisimple and $\Gamma$ a lattice in $G$. Recall that $R=U\rtimes_\rho R_0$ and in this case $H=R_0\times K$.

Consider the projection morphisms $\pr^G_{R}:G\to R$ and $\pr^G_{K}:G\to K$ and denote $\Gamma^R:=\pr^G_{R}(\Gamma)$, $\rmP^K:=\pr^G_{K}(\Gamma)$, which are subgroups of $R$ and $K$, respectively. The former is a lattice in $R$ due to Lemma \ref{lm:gam1lat}. Notice that $\Gamma\cap K$ is a normal subgroup of $\rmP^K$ because the product $R\times K$ is direct. In addition, $\Gamma\cap K$ is finite since it is discrete and contained in the compact group $K$.

\begin{lm}\label{lm:g2va} The group $\rmP^K$ is virtually abelian.
\end{lm}
\begin{proof}We need to show that $\rmP^K$ contains an abelian subgroup of finite index.

For any $\gamma_1\in \Gamma^R$, there exists $\gamma_2\in \rmP^K$  such that $(\gamma_1,\gamma_2)\in\Gamma$. If $\gamma_2'\in\rmP^K$ also verifies $(\gamma_1, \gamma_2')\in \Gamma$, then $(1,\gamma_2^{-1}\gamma_2')\in \Gamma\cap K$ since the product $R\times K$ is direct and $\Gamma$ is a subgroup. This implies that  the following map 
\begin{equation}
\psi:\Gamma^R\to \rmP^K/(\Gamma\cap K),\qquad
\gamma_1\mapsto \psi(\gamma_1):=\gamma_2(\Gamma\cap K)
\end{equation}
is well defined and, one can easily check, that it is a surjective group homomorphism.

Being a subgroup of a solvable group, $\Gamma^R$ is itself solvable and so is its homomorphic image by $\psi$, $\rmP^K/(\Gamma\cap K)$.

As a consequence of Peter-Weyl's Theorem, we know that $K$ embeds into a linear group \cite[Corollary IV.4.22]{KN}. 
In addition, $\rmP^K$ is finitely generated since $\Gamma$ is so, being the fundamental group of a compact manifold. Therefore, the image of $\rmP^K$ under this embedding (which we also denote by $\rmP^K$) is finitely generated. By the Tits alternative \cite{T72}, $\rmP^K$ is either virtually solvable  or contains a free group $F$ on two generators.

Suppose for a contradiction that the latter statement holds. Then the composition $F\hookrightarrow\rmP^K\to\rmP^K/(\Gamma\cap K)$ has kernel $F\cap \Gamma\cap K$, which is finite (being a discrete subgroup of $K$). Since every finite subgroup of a free group is trivial, the projection is injective. Hence $\rmP^K/(\Gamma\cap K)$ contains a free group on two generators, contradicting the fact that it is solvable. 

Consequently, $\rmP^K$ is virtually solvable, that is, it contains a solvable subgroup $D$ of finite index. One has that $\bar D\subset K$ is compact and solvable, and thus its connected component $\bar D_0$ is abelian and of finite index in $\bar D$.

The kernel of the natural map $D\to \bar D/\bar D_0 $ is $D\cap \bar D_0$ which is abelian, and its image is finite. Therefore the abelian group $D\cap \bar D_0$ has finite index in $D$, which has itself finite index in $\rmP^K$.
\end{proof}

\begin{pro}\label{pro:gamexists}
There exists a finite index subgroup $\Lambda\subset \Gamma$ such that $\pr^G_{K}(\Lambda)$ is abelian and $\Lambda\cap K$ is trivial.
\end{pro}
\begin{proof} By Lemma \ref{lm:g2va}, $\rmP^K$ is virtually abelian, i.e. it contains an abelian subgroup $A$ of finite index. Moreover, $\rmP^K$ is finitely generated since $\rmP^K=\pr^G_{K}(\Gamma)$ and $\Gamma$ is the fundamental group of a compact manifold. 

As $A$ is a finite index subgroup of a finitely generated group, it is also finitely generated. Therefore, $A=\Z^k\oplus T$ for some $k\in \N$, where $T$ is a finite group. Notice that $\Z^k$ has finite index in $\rmP^K$.
Set $\Lambda:=(\pr^G_{K})^{-1}(\Z^k)$; clearly, $\Lambda\subset \Gamma$ and  $\pr^G_{K}(\Lambda)$ is abelian. In addition, $\Lambda\cap K\subset \Gamma\cap K$ is discrete and contained in $K$, which is compact, so it is finite. On the other hand, by the definition of $\Lambda$, the intersection $\Lambda\cap K$ is a subgroup of $\Z^k$. Therefore $\Lambda\cap K$  is trivial. 
\end{proof}

As mentioned before, it is enough to prove Theorem \ref{main} for the case where the semisimple part of $G$ is compact and acts trivially on the radical of $G$. Therefore, Theorem \ref{main} is a consequence of the following:

\begin{pro} The characteristic group of a Lie LCP manifold of the form $M=\Gamma\bs (R\times K)$ with $R$ solvable, $K$ compact, and both simply connected, is contained in the radical $R$ of $G$.
\end{pro}
\begin{proof}
Let $(g,\theta)$ be a Lie LCP structure  on $M=\Gamma\bs G$ with $G=R\times K$.
By Proposition \ref{pro:finsg}, the characteristic group of $M$ coincides with that of the quotient of $R\times K$ by a lattice of finite index inside $\Gamma$. Therefore, up to passing to a finite index subgroup and using Proposition \ref{pro:gamexists}, we can assume that 
\begin{equation}\label{triv-ab}
    \Gamma\cap K\mbox{ is trivial and } \pr_K^G(\Gamma) \mbox{ is abelian}. 
\end{equation}

The LCP structure on $M$ induces an LCP structure $(g,\theta,\mmu)$ on the Lie algebra $\mg=\mr\oplus\mk$ of $G=R\times K$. In addition, $\mg$ is unimodular, since $G$ admits a lattice \cite{Mil76}. Consequently, \cite[Theorem 4.2]{dBM24} implies that $(\mg,g,\theta)$ is adapted, that is, $\theta|_\mmu=0$. Moreover, since $\mr$ is the radical of $\mg$, $(\mr, g,\theta)$ is an LCP Lie algebra by \cite[Corollary 4.8]{dBM24}.

We write $R=U\rtimes_\rho R_0$ as in  \eqref{eq:LeviG} and $H=R_0\times K$.  Then, $\rmP=\pr^G_H(\Gamma)$ verifies
\begin{equation}\label{eq:prk}
\rmP\subset \rmP^R\times \rmP^K,
\end{equation}
where  $\rmP^R:=\pr^H_{R_0}(\rmP)$, $\rmP^K:=\pr^H_K(\rmP)$ are the projections on each factor of $R_0\times K$. By \eqref{abc}, $\rmP ^K$ is equal to $\pr^G_K(\Gamma)$, thus it is abelian due to \eqref{triv-ab}.

Let us denote by $\Gamma^R:=\pr^G_R(\Gamma)$, which is a lattice in $R$ due to Lemma \ref{lm:gam1lat}. Then $(g,\theta)$ is a Lie LCP structure on $\Gamma^R\bs R$ with reduced characteristic group $\overline{(\rmP^R)}_0$.

We claim that $\rmP\subset \rmP^R\times \rmP^K$ can be written as the graph of a group homomorphism $f:\rmP^R \to \rmP^K$. Indeed, if $(r,s), (r,s')\in \rmP$, then there exist $u,u'\in U$ such that $(u,r,s),(u',r,s')\in \Gamma$ and thus 
$(u^{-1}u',1,s^{-1}s')\in \Gamma$, which implies that
$(u^{-1}u',1)\in \Gamma^R$ and therefore $(u^{-1}u',1)\in \ker (\pr^R_{R_0})\cap{\Gamma^R}$. However, $\Gamma^R\bs R$ is a Lie LCP manifold, thus by Lemma \ref{lm:2.5}, we get $u=u'$ and therefore $(1,1,s^{-1}s)\in \Gamma\cap K$; the latter intersection is trivial by the choice of $\Gamma$. Hence $s=s'$ and there is a well defined map $f:\rmP^R\to \rmP^K$ such that for every $r\in \rmP^R$, $f(r)$ is the unique $s\in \rmP^K$ verifying $(r,s)\in \rmP$. This is clearly a group morphism and 
\begin{equation}
\label{eq:Pf}
\rmP=\{(r,f(r)):r\in \rmP^R\} \subset \rmP^R\times \Im f.
\end{equation}
as claimed. Moreover, $\pr^{H}_{R_0}(\rmP)=\rmP^R$ and thus $\pr^{H}_{R_0}(\bar\rmP_0)\subset(\overline{\rmP^R})_0$ implying that:
\begin{equation}\label{eq:ppf}
    \rmP\cap \bar\rmP_0 \subset (\rmP^R\cap \overline{(\rmP^R)}_0)\times f(\rmP^R\cap \overline{(\rmP^R)}_0).
\end{equation} 

We consider the lattice $\Gamma^R_0$ of the characteristic group of the Lie LCP manifold $\Gamma^R\bs R$ defined as before by $\Gamma^R_0=\Gamma^R\cap(U\times \overline{(\rmP^R)}_0)$. 
Lemma \ref{lm:pip} applied to $\Gamma^R\bs R$ gives
$\rmP^R\cap \overline{(\rmP^R)}_0=\pr^R_{R_0}(\Gamma_0^R)$. The key point here is that $[\Gamma^R,\Gamma_0^R]$ has finite index in $\Gamma_0^R$ by Proposition \ref{pro:findsol}, so $\pr^R_{R_0}([\Gamma^R,\Gamma_0^R])$ has finite index in $\pr^R_{R_0}(\Gamma_0^R)=\rmP^R\cap \overline{(\rmP^R)}_0$ and thus
\begin{equation}
\label{eq:union}
\rmP^R\cap \overline{(\rmP^R)}_0=\bigcup_{i=1}^m \pr_{R_0}([\Gamma^R,\Gamma_0^R])\cdot p_i
\end{equation}
for some $p_1,\ldots,p_m\in \rmP^R$. In addition, 
\begin{equation}\label{eq:incl}
\pr^R_{R_0}([\Gamma^R,\Gamma_0^R])=[\rmP^R,\rmP_0^R]\subset [\rmP^R,\rmP^R]\subset \rmP^R.
\end{equation}
Moreover, since $\rmP^K$ is abelian, $f$ vanishes on the commutator subgroup of $\rmP^R$. This, together with \eqref{eq:union} and \eqref{eq:incl}, implies $f(\rmP^R\cap \overline{(\rmP^R)}_0)\subset \{f(p_1), \ldots, f(p_m)\}$.

Finally, using \eqref{eq:ppf} we get
\[ \rmP\cap \bar\rmP_0 \subset (\rmP^R\cap \overline{(\rmP^R)}_0)\times \{f(p_1), \ldots, f(p_m)\}.
\]Since $\bar\rmP_0=\overline{\rmP\cap \bar\rmP_0}$ (see Lemma \ref{5.2}) is connected, we get $\bar\rmP_0\subset \overline{\rmP^R\cap \overline{(\rmP^R)}_0}\times \{1\}$. By Lemma \ref{5.2} applied to $\rmP^R$ we thus get 
\begin{equation}\label{pr0}
    \bar\rmP_0\subset \overline{(\rmP^R)}_0\subset R_0.
\end{equation}
\end{proof}

As an immediate consequence of Theorem \ref{main}, we obtain a positive answer to Question \ref{quiz} in the context of Lie LCP manifolds. 
\begin{cor}\label{cor:sc}
Every Lie LCP manifold has simply connected characteristic group. 
\end{cor}
\begin{proof}
By Theorem \ref{main}, the characteristic group of 
a Lie LCP structure on $M=\Gamma\bs G$ is contained in the radical $R$ of $G$. The group $R$ is simply connected, since $G$ is so. It remains to notice that connected Lie subgroups of simply connected solvable Lie groups are simply connected \cite[Theorem 3.18.12]{VAR}. 
\end{proof}

We end up with an algebraic result on the characteristic group of Lie LCP manifolds. 

\begin{cor}\label{cor:comm}
The characteristic group of a Lie LCP manifold $M=\Gamma\bs G$ is contained in the commutator subgroup $G'$ of $G$.
\end{cor}
\begin{proof}  
Note that $G$ is unimodular, since it admits a lattice $\Gamma$. Hence the associated LCP Lie algebra $(\mg,g,\theta,\mmu)$ is adapted. By \cite[Proposition 4.7]{dBM24}, $\mmu$ is contained in the commutator subalgebra $\mg'$ of $\mg$. Therefore, the corresponding connected Lie subgroups satisfy $U\subset G'$. It remains to show that the reduced characteristic group $\bar\rmP_0$ is also contained in $G'$.

Consider the Levi decomposition of $G$, as in \eqref{eq:LeviG} and \eqref{eq:SS}, that is, 
\[G=(U\rtimes R_0)\rtimes (S_1\times K).\]
By Proposition \ref{p=q}, $\bar\rmP_0$ coincides with the reduced characteristic group of the LCP manifold $\Sigma \bs(R\times K)$. In addition, by \eqref{pr0} we have $\bar\rmP_0\subset \overline{(\rmP^R)}_0$ so it suffices to show that the reduced characteristic group $\overline{(\rmP^R)}_0$ of the Lie LCP solvmanifold $\Gamma^R\bs R$ is contained in $R'\subset G'$. In order to simplify notation, we can thus assume that $G=R$ is solvable.

Recall that $\pr^G_H$ is an isomorphism from $\Gamma$ to $\rmP$, so by Proposition 4.10, $[\rmP,\pr^G_H(\Gamma_0)]$ has finite index in $\pr^G_H(\Gamma_0)$, which according to Lemma 4.6 is equal to $\rmP\cap\bar\rmP_0$. Therefore, 
there exist elements $\gamma_1,\ldots,\gamma_k\in \rmP\cap\bar\rmP_0$ such that $\rmP\cap\bar\rmP_0=\cup_{i=1}^k[\rmP,\rmP\cap\bar\rmP_0]\gamma_i$. By passing to the closure and using Lemma 2.6 we get 
\begin{equation}
    \bar\rmP_0=\overline{\rmP\cap\bar\rmP_0}\subset\cup_{i=1}^k[\bar\rmP,\overline{\rmP\cap\bar\rmP_0}]\gamma_i\subset \cup_{i=1}^k G'\gamma_i
    \label{eq:PPG}.
\end{equation}

Note that $G'$ is closed since $G$ is simply connected, and connected \cite[Theorem 3.18.8]{VAR}, so each $G'\gamma_i$ in \eqref{eq:PPG} is open and closed. Therefore,  $\bar\rmP_0\subset G'$ which is the connected component of the identity in the union on the right hand side.
\end{proof}

\bibliographystyle{plain}
\bibliography{biblio}

@Article{Wu88,
 Author = {Wu,T.~S.},
 Title = {A note on a theorem on lattices in {Lie} groups},
 FJournal = {Canadian Mathematical Bulletin},
 Journal = {Can. Math. Bull.},
 ISSN = {0008-4395},
 Volume = {31},
 Number = {2},
 Pages = {190--193},
 Year = {1988},
 Language = {English},
 DOI = {10.4153/CMB-1988-029-8},
 zbMATH = {4065324},
 Zbl = {0653.22009}
}

@Book{KN,
author = {Knapp,A.},
   TITLE = {Lie groups beyond an introduction},
    SERIES = {Progress in Mathematics},
    VOLUME = {140},
 PUBLISHER = {Birkh\"auser Boston Inc.},
   ADDRESS = {Boston, MA},
      YEAR = {2002},
note = {2nd Edition},}

@Book{VAR,
author = {Varadarajan,V.~S.},
TITLE = {Lie groups, {L}ie algebras, and their representations},
    SERIES = {Graduate Texts in Mathematics},
    VOLUME = {102},
      NOTE = {Reprint of the 1974 edition},
 PUBLISHER = {Springer-Verlag},
   ADDRESS = {New York},
      YEAR = {1984},}

@Book{H,
author = {Hatcher,A.},
title = {Spectral Sequences in Algebraic Topology},
publisher = {http://www.math.cornell.edu/~hatcher/ },
year = {},
series = {},
}

@Book{RA,
author = {Raghunathan,M.~S.},
title = {Discrete subgroups of {L}ie Groups},
publisher = {Springer},
year = {1972},
series = {},
}

@Article{AD,
    Author = {{Andrada},A. and {Dotti},I.},
    Title = {{Conformal Killing-Yano 2-forms.}},
    Journal = {{Differ. Geom. Appl.}},
    Volume = {58},
    Pages = {103--119},
    Year = {2018},
}

@Article{KO,
author = {Kostant,B.},
title = {{L}ie Algebra Cohomology and the Generalized Borel-Weil Theorem},
year = {1961},
journal = {Ann. of Math.},
volume = {72},
pages = {329--390},
}

@Article{Mil76,
    Author = {{Milnor},J.},
    Title = {{Curvatures of left invariant metrics on Lie groups.}},
    Journal = {{Adv. Math.}},
    Volume = {21},
    Pages = {293--329},
    Year = {1976},
}

@article {Gaud94,
    AUTHOR = {Gauduchon,P.},
     TITLE = {Structures de {W}eyl-{E}instein, espaces de twisteurs et
              vari\'et\'es de type {$S^1\times S^3$}},
   JOURNAL = {J. Reine Angew. Math.},
  FJOURNAL = {Journal f\"ur die Reine und Angewandte Mathematik. [Crelle's
              Journal]},
    VOLUME = {469},
      YEAR = {1995},
     PAGES = {1--50},
}

@Article{Fl24,
    Author = {Flamencourt, B.},
    Title = {{Locally conformally product structures}},
   Journal  = {{Internat. J. Math.}},
     Volume = {35},
     Number = {5},
      Pages = {2450013},
   Year = {2024},
}

@Article{Fl25,
    Author = {Flamencourt, B.},
    Title = {{The characteristic group of locally conformally product structures}},
   Journal  = {{Ann. Mat. Pura Appl.}},
 Volume = {204}, 
Pages = {189--211},
   Year = {2025},
}

@Article{FM25,
    Author = {Flamencourt, B. and Moroianu, A},
    Title = {{On decomposable LCP structures}},
   Journal  = {{Geom. Dedicata}},
 Volume = {219}, 
number = {1},
Pages = {11},
   Year = {2025},
}

@Book{We23,
 Author = {{Weyl},H.},
 Title = {{Raum, Zeit, Materie. Vorlesungen \"uber allgemeine Relativit\"atstheorie.~5.~Aufl.}},
 Year = {1923},
 Language = {German},
 Publisher = {{Berlin: J.~Springer, VIII u.~338 S.}},
 
}

@Article{BM2016,
 Author = {Belgun,F. and Moroianu,A.},
 Title = {On the irreducibility of locally metric connections},
 Journal = {J. Reine Angew. Math.},
 Volume = {714},
 Pages = {123--150},
 Year = {2016},
 }

@Article{MN2015,
 Author = {Matveev,V. and Nikolayevsky,Y.},
 Title = {A counterexample to {Belgun-Moroianu} conjecture},
 Journal = {C. R. Math. Acad. Sci. Paris},
 Volume = {353},
 Pages = {455--457},
 Year = {2015},
 }

@Article{Kourganoff,
 Author = {Kourganoff,M.},
 Title = {Similarity structures and de {Rham} decomposition},
 Journal = {Math. Ann.},
 Volume = {373},
 Pages = {1075--1101},
 Year = {2019},
 }

@Article{MN2017,
 Author = {Matveev,V. and Nikolayevsky,Y.},
 Title = {Locally conformally {Berwald} manifolds and compact quotients of reducible manifolds by homotheties},
 Journal = {Ann. Inst. Fourier (Grenoble)},
 Volume = {67},
 Number = {2},
 Pages = {843--862},
 Year = {2017},
 }

@Article{AdBM,
 Author = {Andrada,A. and del Barco, V. and Moroianu, A.},
 Title = {Locally conformally product structures on solvmanifolds},
 Journal = {{Ann. Mat. Pura Appl.}},
 Volume = {203},
 Pages = {2425--2456},
 Year = {2024},
 }

@article{dBM24,
 Author = {del Barco, V. and Moroianu, A.},
 Title = {The structure of locally conformally product {L}ie algebras},
 journal = {Internat. J. Math},
volume = {36},
number = {14}, 
pages = {2550061},
 Year = {2025},
 }

@Article{T72,
 Author = {Tits,J.},
 Title = {Free subgroups in linear groups},
 Journal = {{ J. Algebra}},
  Volume = {20},
  Pages = {250--270},
 Year = {1972},
 }

\end{document}